\documentclass[11pt]{amsproc}

\usepackage{mathtext}
\usepackage[cp1251]{inputenc}

\usepackage{bm}

\usepackage[dvips]{graphicx}
\usepackage{amsmath}
\usepackage{amssymb}
\usepackage{amsxtra}

\def\N{{{\Bbb N}}}
\def\Z{{{\Bbb Z}}}
\def\T{{{\Bbb T}}}
\def\R{{\Bbb R}}
\def\C{{\Bbb C}}
\def\l{{\lambda }}
\def\a{{\alpha }}
\def\D{{\Delta }}

\def\a{{\alpha}}
\def\b{{\beta}}
\def\d{{\delta}}
\def\e{{\varepsilon}}
\def\s{{\sigma}}
\def\vp{{\varphi}}
\def\t{{\theta }}
\def\g{{\gamma }}

\def\j{{\bm j}}
\def\k{{\bm k}}

\def\x{{\bm x}}

\def\ee{{\bm e}}

\def\elll{{\bm \ell}}

\newcommand{\h}{\widehat}
\newcommand{\w}{\widetilde}

\def\){\right)}
\def\({\left(}

\def\spec{\operatorname{spec}}
\def\mix{\operatorname{mix}}

\numberwithin{equation}{section}

\newtheorem{theorem}{Theorem}[section]

\newtheorem{corollary}[theorem]{Corollary}
\newtheorem{lemma}[theorem]{Lemma}

\newtheorem{proposition}[theorem]{Proposition}
\newtheorem{remark}[theorem]{Remark}

\newtheorem{example}[theorem]{Example}

\par

\begin{document}

\title[]{Sparse grid approximation in weighted Wiener spaces}

\author[Yurii
Kolomoitsev]{Yurii
Kolomoitsev$^{\text{a, 1, 2}}$}
\address{Institute for Numerical and Applied Mathematics, G\"ottingen University, Lotzestr. 16-18, 37083 G\"ottingen, Germany}
\email{kolomoitsev@math.uni-goettingen.de}

\author[Tetiana
Lomako]{Tetiana
Lomako$^{\text{a, 1}}$}
\address{Institute for Numerical and Applied Mathematics, G\"ottingen University, Lotzestr. 16-18, 37083 G\"ottingen, Germany}
\email{tlomako@yandex.ua}
\author[Sergey Tikhonov]{Segey Tikhonov$^\text{b, 3}$}
\address{Centre de Recerca Matem\`atica, Campus de Bellaterra, Edifici C 08193
Bellaterra, Barcelona, Spain; ICREA, Pg. Llu\'is Companys 23, 08010
Barcelona, Spain, and Universitat Aut\'onoma de Barcelona}
\email{stikhonov@crm.cat}

\thanks{$^\text{a}$Institute for Numerical and Applied Mathematics, G\"ottingen University, Lotzestr. 16-18, 37083 G\"ottingen, Germany}

\thanks{$^\text{b}$Centre de Recerca Matem\`atica, Campus de Bellaterra, Edifici C 08193
Bellaterra, Barcelona, Spain; ICREA, Pg. Llu\'is Companys 23, 08010
Barcelona, Spain, and Universitat Aut\'onoma de Barcelona}


\thanks{$^1$Supported by the German Research Foundation, project KO 5804/1-2}
\thanks{$^2$Support by the German Research Foundation in the framework of the RTG 2088}
\thanks{$^3$Supported by
PID2020-114948GB-I00,  2017 SGR 358, AP08856479
and  the CERCA Programme of the Generalitat de Catalunya.
 Also, supported by the Spanish State Research Agency, through the Severo Ochoa and Mar\'ia de Maeztu Program for Centers and Units of Excellence in R\&D (CEX2020-001084-M). The author thanks CERCA Programme/Generalitat de Catalunya for institutional support}


\thanks{E-mail address: stikhonov@crm.cat}

\date{\today}
\subjclass[2010]{41A25, 41A63, 42A10, 42A15, 41A58, 	41A17, 	42B25, 42B35} \keywords{Sparse grid, weighted Wiener spaces, quasi-interpolation operators,  Kantorovich operators, Smolyak algorithm, Littlewood--Paley-type characterizations}

\begin{abstract}
 We study approximation properties of multivariate periodic functions
from weighted Wiener spaces by sparse grids methods constructed with
the help of quasi-interpolation operators. The class of such operators
includes classical interpolation and sampling operators, Kantorovich-type
operators, scaling expansions associated with wavelet constructions,
and others. We obtain the rate of convergence of the corresponding
sparse grids methods in weighted Wiener norms as well as analogues of the Littlewood–Paley-type characterizations in terms of families of quasi-interpolation operators.
 \end{abstract}

\maketitle

\section{Introduction}

In many applied problems one needs to approximate high-dimensional functions in smooth function spaces. As it is known from previous research, traditional numerical methods such as the interpolation with tensor product grids suffer from the so-called "curse of dimensionality". In other words, the computation time of such methods grows dramatically with the number of variables and the problem becomes intractable already for mild dimensions. One of the means to overcome these obstacles is to employ different sparse grids approximation methods and to impose additional assumption on smoothness. Typically, one  assumes that a function belongs to a certain mixed smoothness Sobolev or Besov space (see, e.g.,~\cite{BG04}, \cite{DTU18}).

In this paper, we consider approximation methods that are based on generalized sparse grids (see, e.g.,~\cite{BG04}, \cite{GK00}).
Recall that for given parameters $T\in [-\infty,1)$, $n\in\N$, and a family of univariate operators
$Y=(Y_j)_{j\in \Z_+}$, a sparse grid approximation method is defined as follows:
\begin{equation}\label{P}
  P_{n,T}^Y=\sum_{\j\in\D(n,T)} \eta_\j^Y,\quad \eta_\j^Y=\prod_{i=1}^d (Y_{j_i}^i-Y_{j_i-1}^i),
\end{equation}
where
\begin{equation*}
 \D(n,T)=\{\k\in \Z_+^d\,:\, |\k|_1-T|\k|_\infty\le (1-T)n\}
\end{equation*}
and $Y_j^i$ denotes the univariate operator $Y_j$ acting on functions in the variable $x_i$ and $Y_{-1}=0$.
The most well studied case of the family $Y$ is the classical Lagrange interpolation operators $I=(I_j)_{j\in\Z_+}$ given by
\begin{equation*}
  I_j(f)(x)=2^{-j}\sum_{k=-2^{j-1}}^{2^{j-1}-1} f(x_k^j)\mathcal{D}_j(x-x_k^j),
\end{equation*}
where $x_k^j=\frac{\pi k}{2^{j-1}}$ and $\mathcal{D}_j(x)=\sum_{\ell=-2^{j-1}}^{2^{j-1}-1} e^{{\rm i}\ell x}$ is the Dirichlet kernel.
The corresponding sparse grid for a given level $n$ is then 
$$
\Gamma({n,T})=\bigcup_{\j\in\D(n,T)}\mathcal{I}_{j_1}\times\dots\times \mathcal{I}_{j_d},
$$
where $\mathcal{I}_j=\{x_k^j\,:\,k=-2^j,\dots,2^j-1\}$, i.e.,  $P_{n,T}^I f(y)=f(y)$ for all $y\in \Gamma(n,T)$ and $f\in C(\T^d)$.
Here, the case $T=-\infty$ corresponds to the interpolation on the full tensor grid; the case $T=0$ represents interpolation on the Smolyak grid, which is also called the regular sparse grid (see~\cite{Sm63}, see also~\cite[Ch.~5]{DTU18}); and the case $0<T<1$ resembles the so-called energy-norm based sparse grids (see~\cite{BG04}, \cite{GH14}).

One of the important characteristics of a sparse grid is its cardinality. Note that (see, e.g.,~\cite{GK00})
\begin{equation}\label{card}
  {\rm card}\,\Gamma({n,T})\lesssim \sum_{\k\in \D(T,n)}2^{|\k|_1}\lesssim \left\{
                                                                             \begin{array}{ll}
                                                                               2^{n}, & \hbox{if $0<T<1$,} \\
                                                                               2^n n^{d-1}, & \hbox{if $T=0$,} \\
                                                                               2^{(\frac{1-T}{1-T/d})n}, & \hbox{if $T<0$,} \\
                                                                               2^{dn}, & \hbox{if $T=-\infty$}
                                                                             \end{array}
                                                                           \right.
\end{equation}
and the same upper bound holds for the number of frequencies of the polynomial $P_{n,T}^I f$. Thus, the most interesting cases for approximation with the algorithm $P_{n,T}^I$ is when $0\le T<1$ since in this case, for the number of elements needed to construct the corresponding algorithm, there is no exponential growth with the dimension $d$. Let us consider this case in more detail.



Approximation properties of the operators $P_{n,T}^I$ have been mainly studied in the case $T=0$, which corresponds to the classical Smolyk grids (see, e.g.,~\cite{AT97}, \cite{BU17}, \cite{D85}, \cite{Ha92}, \cite{SU07}, \cite{Sp00}, \cite{T86}, \cite{U08}; see also the book~\cite[Ch.~4 and Ch.~5]{DTU18}).
As an example, we mention the following $L_q$-error estimates for the Smolyak algorithm $P_{n,0}^I$ in the case of
approximation of functions from the Sobolev space $\mathbf{W}_{p}^\a(\T^d)$ of dominating mixed smoothness $\a>0$ (see, e.g.,~\cite{SU07} and~\cite[Chapters~4 and~5]{DTU18}): if  $1< p,q<\infty$ and $\a>\max\{1/p,1/2\}$, then
\begin{equation}\label{I2}
   \sup_{f\in U\mathbf{W}_{p}^\a}\|f-P_{n,0}^I(f)\|_{L_q(\T^d)} \asymp \left\{
                                                                                \begin{array}{ll}
                                                                                   2^{-\a n}n^{\frac{d-1}2}, & \hbox{if $p\ge q$,} \\
                                                                                   2^{-(\a-1/p+1/q)n}, & \hbox{if $q>p$,}
                                                                                \end{array}
                                                                              \right.
\end{equation}
where $U\mathbf{W}_{p}^\a$ denotes the unit ball in the space $\mathbf{W}_{p}^\a(\T^d)$.
Similar estimates (for $T=0$) in the weighted Wiener spaces (or Korobov spaces)
were obtained in~\cite{DS89}, \cite{Ha92}, \cite{Sp00}.

In the case $0<T<1$, the approximations by operators $P_{n,T}^I$ have been mostly investigated for functions from the so-called generalized mixed smoothness (or hybrid smoothness) Sobolev space
$$
H^{\a,\b}(\T^d):=\left\{f\in L_2(\T^d)\,:\, \sum_{\k\in \Z^d}\prod_{j=1}^d (1+|k_j|)^{2\a}(1+|\k|)^{2\b}|\h f(\k)|^2<\infty\right\},
$$
where the parameter $\b$ governs for the isotropic smoothness, whereas $\a$  reflects the smoothness in the dominating mixed sense.
Herewith, the approximation error is estimated in the metric of the classical isotropic Sobolev space $H^\g(\T^d)=H^{0,\g}(\T^d)$, see, e.g.,~\cite{BG04}, \cite{GH14}, \cite{GH20}, \cite{BDSU16}. In particular, we mention a general result obtained in the recent paper~\cite{GH20}:
let $\a\ge 0$, $\b\ge 0$, $\g-\b<\a$, $\a+\frac\b d>\frac12$. Then, for all $f\in H^{\a,\b}(\T^d)$ and $n\in \N$, one has
  \begin{equation}\label{gr}
       \|f-P_{n,T}^I f\|_{H^\g(\T^d)}\lesssim \Omega_I(n)\|f\|_{H^{\a,\b}(\T^d)},
  \end{equation}
where\footnote{Note that there are typos in formula~(21) and related estimates in~\cite{GH20}. See also~\cite[Lemma~8]{GH14}.}
\begin{equation*}
  \begin{split}
      \Omega_I(n)
=\left\{
                                                                         \begin{array}{ll}
                                                                            \displaystyle 2^{-\(\a-(\g-\b)-\(\a T-(\g-\b)\)\frac{d-1}{d-T}\)n}n^{\frac{d-1}2}, & \hbox{$T\ge\frac{\g-\b}{\a}$,} \\
                                                                           \displaystyle  2^{-(\a-(\g-\b))n}, & \hbox{$T<\frac{\g-\b}{\a}$,}
                                                                          \end{array}
                                                                        \right.
   \end{split}
\end{equation*}
which again shows the importance of the case $0\le T\le \frac{\g-\b}{\a}<1$, cf.~\eqref{card}.


Along with the classical interpolation operators $I=(I_j)_{j\in\Z_+}$ one also considers the family of the partial sums of Fourier series (see, e.g.,~\cite{AT97}, \cite{D86}), families of convolution type operators (see, e.g.,~\cite{U08}, \cite{SU07}), quasi-interpolation operators based on scaled $B$-splines with integer knots (see, e.g.,~\cite{D11}, \cite{D16}, \cite{D18}).

In this paper, as a family $Y$ in~\eqref{P}, we make use of the general quasi-interpolation operators
defined by 
\begin{equation}\label{Q0-}
Q_j(f,\vp_j,\w\vp_j)(x)=2^{-j}\sum_{k=-2^{j-1}}^{2^{j-1}-1} (f*\w\vp_j)(x_k^j)\vp_j(x-x_k^j),\quad j\in\Z_+,
\end{equation}
where $(\vp_j)_{j\in \Z_+}$ is a family of univariate trigonometric polynomials and $(\w\vp_j)_{j\in \Z_+}$ is a family of functions/distributions on $\T$.
Note that in the non-periodic case approximation properties of operators of such type in various function spaces (classical and weighted $L_p$, Sobolev, Besov, Wiener) have been studied, for example, in the works~\cite{FK07}, \cite{JZ95}, \cite{Jia10}, \cite{KS21}, \cite{Ky96}. It worths noting that the operators~\eqref{Q0-} can be successfully employed in such applied problems, where the data contains some noise and the functional information is provided by other means than point evaluation (averages,  divided differences, etc.), see, e.g.,~\cite{CSV20}, \cite{ZWS12}.

In the periodic case, an analog of estimate~\eqref{I2} for any $\a>0$ has been recently established in~\cite{K21} for the approximation processes $P_{n,0}^Q$ with $Q=(Q_j)_{j\in\Z_+}$ defined by~\eqref{Q0-}. The corresponding proof is essentially based on the results
 from~\cite{KP21}, where under different compatibility conditions on $(\vp_j)_{j\in \Z_+}$ and $(\w\vp_j)_{j\in \Z_+}$, $L_p$-error of approximation by the operators $Q_j$ were obtained. Similar results in weighted Wiener spaces and $L_2(\T)$ have been derived in~\cite{KKS20} and~\cite{JBU02}, correspondingly.

The goal of the present work is to establish analogues of error estimate~\eqref{gr} for sparse grid approximation methods constructed using general quasi-interpolation operators. Comparing our findings with the previously known results, we stress two important differences.   Firstly, we build the approximation schemes using quasi-interpolation operators~\eqref{Q0-} rather than the classical interpolation operators $(I_j)_{j\in\Z_+}$ constructed using the Dirichlet kernel and the values of a function at sets of equidistant interpolation nodes as in~\cite{BDSU16}, \cite{GH14}, \cite{GH20}. Secondly, we work with a more general scale of spaces, namely with the weighted Wiener spaces $A_p^{\a,\b}(\T^d)$ rather than with the Sobolev spaces $H^{\a,\b}(\T^d)$, which correspond to the case $p=2$. We would like to stress that possibility to vary the family $(\w\vp_j)_{j\in\Z_+}$ allows us to prove the results under essentiality less restrictive conditions on the parameters $\a$ and $\b$.
%
%

The paper is organized as follows. In
Section~2 we introduce basic notations, define isotropic, mixed, and hybrid weighted Wiener spaces and general quasi-interpolation operators. Section~3 is devoted to auxiliary results. In particular, we prove the following useful estimate
$$
\|f-Q_j(f,\vp_j,\w\vp_j)\|_{A_q^\g(\T)}\lesssim 2^{-j\min(\a-\g,s)}\|f\|_{A_q^\a(\T)},
$$
see Lemma~\ref{leKKS}.
In Section~4 we establish our main tools, the so-called "discrete" Littlewood-Paley type characterizations.
In Section~5 we prove our main results: we consider approximation in the isotropic Wiener space $A_q^\g(\T^d)$ (in Subsection~5.1) and in the mixed Wiener spaces $A_{q,\mix}^\g(\T^d)$ (in Subsection~5.2). In Subsection~5.3 we discuss the sharpness of the obtained results. In Section~6 we consider the specific sets of parameters, where our main results (Theorems~\ref{th1} and~\ref{th1+}) provide the most effective error estimates with respect to the approximation rate and the number of degrees of freedom.

\section{Notation. Function spaces and operators}

\subsection{Basic notation}

 In what follows, $\Z_+^d=\{\x\in\Z^d:~x_i\geq~{0}, i=1,\dots, d\}$ and
    $\T^d=\R^d\slash2\pi\Z^d$ is the $d$-dimensional torus.
    Further, for vectors  $\x = (x_1,\dots, x_d)$ and
    $\k =(k_1,\dots, k_d)$ in $\R^d$, we denote
     $(\x, \k)=x_1k_1+\dots+x_dk_d$.
		If ${\bm j}\in\Z^d_+$,  we set
    $|\bm{j}|_1=\sum_{k=1}^d j_k$, $|\bm{j}|_\infty=\max_{k=1}^d j_k$, and $2^{\j}=(2^{j_1},\dots,2^{j_d})$. For $1\le p\le\infty$, $p'$ is given by
$\frac{1}{p}+\frac{1}{p'}=1.$ For $1\le p,q\le \infty$, we set $\s_{p,q}=\(\frac1q-\frac1p\)_+$.

If $f\in L_1(\T^d)$, then
$$
\h f(\k)=(2\pi)^{-d}\int_{\T^d} f(\x){e}^{-{\rm i}(\k,\x)}{\rm d}\x,\quad \k\in\Z^d,
$$
denotes the $k$-th Fourier coefficient of $f$.
As usual, the convolution of integrable functions $f$ and $g$ is given by
$$
(f*g)(\x)=(2\pi)^{-d}\int_{\T^d} f(\x-{\bm  t})g({\bm  t}){\rm d}{\bm  t}.
$$
By $\mathcal{T}_\j^d$, $\j\in \Z_+^d$, we denote the following set of trigonometric polynomials:
$$
\mathcal{T}_\j^d={\rm span}\left\{e^{{\rm i}(\k,\x)}\,:\,\k\in  D_{j_1}\times\dots\times D_{j_d}\right\},
$$
where
$$
D_j=[-2^{j-1},2^{j-1})\cap \Z.
$$

Let ${\mathcal{D}} = C^{\infty}(\T)$ be the space of infinitely differentiable functions on $\T$.
The linear space of periodic distributions (continuous linear functionals on ${\mathcal{D}}$) is denoted by ${\mathcal{D}}'$.
It is known (see, e.g.,~\cite[p. 144]{ST87}) that any periodic distribution ${\w\vp}$ can be expanded in a weakly convergent (in ${\mathcal{D}}'$) Fourier series
\begin{equation*}
{\w\vp}(x) = \sum_{k\in\Z} \h{\w\vp}(k) {e}^{{\rm i} kx},
\end{equation*}
where the sequence $(\h{\w\vp}(k))_k$ has at most polynomial growth.

Throughout the paper, we use the notation
$
\, A \lesssim B,
$
with $A,B\ge 0$, for the estimate
$\, A \le C\, B,$ where $\, C$ is a positive constant independent of
the essential variables in $\, A$ and $\, B$ (usually, $f$, $j$, and $n$).
If $\, A \lesssim B$
and $\, B \lesssim A$ simultaneously, we write $\, A \asymp B$ and say that $\, A$
is equivalent to $\, B$.
For two function spaces
$\, X$ and $\, Y,$ we will use the notation
$
\, Y \hookrightarrow X
$
if
$\, Y \subset X$ and $\, \| f\|_X \lesssim \| f\|_Y$ for all $\, f \in
Y.$ The unit ball in some normed vector space $X$ is denoted by $UX$.

\subsection{Weighted Wiener-type spaces} We will employ the following function spaces with the parameters $\a, \b\in\R$, and $0<q\le \infty$.

\smallskip

\noindent $\bullet$ \emph{The periodic (isotropic) Wiener space} $A_{q}^\a(\T^d)$ is the collection of all $f\in L_1(\T^d)$ such that
$$
\|f\|_{A_{q}^\a(\T^d)}':=\(\sum_{\k\in \Z^d}(1+|\k|)^{q\a}|\h f(\k)|^q\)^{1/q}, \quad q<\infty,
$$
$$
\|f\|_{A_{\infty}^\a(\T^d)}':=\sup_{\k\in \Z^d}(1+|\k|)^{\a}|\h f(\k)|,\quad q=\infty.
$$
In the case $\a=0$, we use the following standard notation
$
A_q(\T^d)=A_q^0(\T^d).
$

\smallskip

\noindent $\bullet$ \emph{The periodic (mixed) Wiener space} $A_{q,\mix}^\a(\T^d)$ is the collection of all $f\in L_1(\T^d)$ such that
$$
\|f\|_{A_{q,\mix}^\a(\T^d)}':=\(\sum_{\k\in \Z^d}\prod_{j=1}^d (1+|k_j|)^{q\a}|\h f(\k)|^q\)^{1/q}, \quad q<\infty,
$$
$$
\|f\|_{A_{\infty,\mix}^\a(\T^d)}':=\sup_{\k\in \Z^d}\prod_{j=1}^d (1+|k_j|)^{\a}|\h f(\k)|,\quad q=\infty.
$$

\noindent $\bullet$ \emph{The periodic (hybrid) Wiener space} $A_q^{\a,\b}(\T^d)$ is the collection of all $f\in L_1(\T^d)$ such that
$$
\|f\|_{A_{q}^{\a,\b}(\T^d)}':=\(\sum_{\k\in \Z^d}\prod_{j=1}^d (1+|k_j|)^{q\a}(1+|\k|)^{q\b}|\h f(\k)|^q\)^{1/q}, \quad q<\infty,
$$
$$
\|f\|_{A_{\infty}^{\a,\b}(\T^d)}':=\sup_{\k\in \Z^d}\prod_{j=1}^d (1+|k_j|)^{\a} (1+|\k|)^{\b}|\h f(\k)|,\quad q=\infty.
$$

\begin{remark}\label{remb}
$(i)$  It is easy to see that for any $\a>0$ and $1\le q\le \infty$, the following embeddings hold:
$$
A^{d\a}_q(\T^d) \hookrightarrow A^{\a}_{q,\mix}(\T^d) \hookrightarrow A^{\a}_q(\T^d).
$$

$(ii)$ Note that the isotropic Wiener space $A_q^\a(\T^d)$ coincides with the corresponding periodic Sobolev space $H^\a(\T^d)$ in the case $q=2$. The same holds for the mixed or hybrid Wiener spaces.

$(iii)$ For more information about the weighted Wiener spaces $A_q^{\a,\b}(\T^d)$, see papers~\cite{B89}, \cite{Dya92},  \cite{LST12}, \cite{Mu58}, \cite{NS20}.
\end{remark}

As usual, for $f\in L_1(\T^d)$, we define the diadic blocks $\d_\k(f)$, $\k\in \Z_+^d$, by
$$
\d_\k(f)(x)=\sum_{\k\in \mathcal{P}_\k} \h f(\k)e^{{\rm i}(\k,\x)},
$$
where
\begin{equation*}
\mathcal{P}_\k:=P_{k_1}\times\dots\times P_{k_d},
\end{equation*}
$P_j=\{\ell \in \Z\,:\, 2^{j-1}\le |\ell|<2^j\}$ for $j>0$, and  $P_0=\{0\}$.

Recall that for all $f\in L_p(\T^d)$, $1<p<\infty$, the Littlewood-Paley decomposition reads as follows
$$
f=\sum_{\elll\in \Z_+^d}\d_\elll(f).
$$

The next lemma is a simple consequence of the definition of the space $A_q^{\a,\b}(\T^d)$.

\begin{lemma}\label{le1}
  Let $0< q\le\infty$, $\a\ge 0$, and let $\b\in \R$ be such that $\a+\b\ge 0$. Then
  $$
   A_q^{\a,\b}(\T^d)=\left\{ f\in L_1(\T^d)\,:\, \|f\|_{A_q^{\a,\b}(\T^d)}:=\bigg(\sum_{\k\in \Z_+^d} 2^{q(\a|\k|_1+\b|\k|_\infty)}\|\d_\k(f)\|_{A_q(\T^d)}^q\bigg)^{1/q}<\infty\right\}
  $$
  with the usual modification in the case $q=\infty$ in the sense of equivalent norms.
\end{lemma}

\subsection{Quasi-interpolation operators}
Consider a family of general univariate quasi-interpolation operators $Q=(Q_j)_{j\in \Z_+}$ given by 
\begin{equation*}
  Q_j(f,\vp_j,\w\vp_j)(x)=2^{-j}\sum_{k\in D_j} (f*\w\vp_j)(x_k^j)\vp_j(x-x_k^j), \quad x_k^j=\frac{\pi k}{2^{j-1}},
\end{equation*}
where $(\vp_j)_{j\in \Z_+}$ is a family of univariate trigonometric polynomials in $\mathcal{T}_j^1$, $(\w\vp_j)_{j\in \Z_+}$ is a family of functions/distributions, and the convolution $f*\w\vp_j$ is defined in some suitable way for any $j\in \Z_+$.

Below, we assume that the following conditions on $(\vp_j)_{j\in \Z_+}$ and $(\w\vp_j)_{j\in \Z_+}$ hold:

\smallskip

    $\bullet$ \emph{The growth condition of order} $N\ge 0$ for the Fourier coefficients of $\w\vp_j$:
\begin{equation}\label{c1}
 \begin{split}
| \h{\w\vp_j}(\ell)| & \le C_{\w\vp} (1+|2^{-j}\ell|^{N}) \quad\text{for all}\quad  \ell\in\Z, \quad  j\in\Z_+.
 \end{split}
\end{equation}

\medskip

$\bullet$ \emph{The uniform boundedness condition} for the Fourier coefficients of $\vp_j$:
\begin{equation}\label{c2}
|\h{\vp_j}(\ell)| \le C_{\vp} \quad\text{for all}\quad  \ell\in\Z, \quad  j\in\Z_+.
\end{equation}

\medskip

$\bullet$ \emph{The compatibility condition  of order} $s> 0$ for $\vp_j$ and $\w\vp_j$:
\begin{equation}\label{c3}
|1-\h{\vp_j}(\ell){\h{\w\vp_j}(\ell)}| \le C_{\vp,\w\vp,s}  |2^{-j} \ell|^s \quad\text{for all}\quad \ell\in D_j, \quad  j\in\Z_+.
\end{equation}

\medskip

Let us consider two important classes of quasi-interpolation operators and examine the above conditions.

\begin{example}\label{ex1}
\normalfont\emph{Quasi-interpolation sampling operators} are defined by
\begin{equation}\label{s}
  S_j(f,\vp_j)(x)=2^{-j}\sum_{k\in D_j} \(\sum_{|\nu|\le m}a_{\nu,j} f^{(r_\nu)}(x_{k-\nu}^j)\)\vp_j(x-x_k^j),
\end{equation}
where $a_{\nu,j}\in \C$, $r_\nu\in \Z_+$, and $\vp_j\in \mathcal{T}_j^1$. Note that $S_j(f,\vp_j)=Q_j(f,\vp_j,\w\vp_j)$ with
$$
\w\vp_j(x)=\sum_{|\nu|\le m}a_{\nu,j} \delta^{(r_\nu)}(x-x_{\nu}^j)\sim \sum_{\ell\in \Z}\(\sum_{|\nu|\le m}a_{\nu,j} ({\rm i}\ell)^{r_\nu}e^{-{\rm i}\ell x_\nu^j}\)e^{{\rm i}\ell x}.
$$
One can see that condition \eqref{c1} with $N=\max_{|\nu|\le m} r_\nu$ is satisfied if
$$
\sup_{j\in\Z_+}\sum_{|\nu|\le m}2^{r_\nu j}|a_{\nu,j}|<\infty
$$
and condition \eqref{c3} with $s>0$ is fulfilled if
$$
\bigg|1-\h{\vp_j}(\ell)\sum_{|\nu|\le m}\overline{{a_{\nu,j}}} (-{\rm i}\ell)^{r_\nu}e^{{\rm i}\ell x_\nu^j}\bigg|\le c|2^{-j}\ell|^s\quad\text{for all}\quad \ell\in D_j,\quad j\in \Z_+.
$$

\noindent 1a) A particular example of~\eqref{s} is the classical Lagrange interpolation operator
\begin{equation}\label{I+}
  I_j(f)(x)=2^{-j}\sum_{k\in D_j}f(x_k^j)\mathcal{D}_j(x-x_k^j),
\end{equation}
where
$$
\mathcal{D}_j(x)=\sum_{\ell\in D_j}e^{{\rm i}\ell x}
$$
is the Dirichlet kernel. Note that $I_j(f)=S_j(f,\vp_j)$ with $\vp_j=\mathcal{D}_j$, $m=0$, $a_{0,j}=1$, and $r_0=0$. In this case it is easy to see that conditions~\eqref{c1}--\eqref{c3} are fulfilled for $N=0$ and any $s>0$.

\noindent 1b) As an example of quasi-interpolation operators that are generated by an average sampling instead of the exact samples of $f$, we consider
\begin{equation}\label{A}
  A_j(f)(x)=2^{-j}\sum_{k\in D_j} \lambda_j(f)(x_k^j)\mathcal{D}_j(x-x_k^j),
\end{equation}
where
$$
\lambda_j(f)(x)=\frac14\(f\(x-\frac{\pi}{2^j}\)+2f(x)+f\(x+\frac{\pi}{2^j}\)\).
$$
We have that $A_j(f)=Q_j(f,\vp_j,\w\vp_j)$ with $\vp_j(x)=\mathcal{D}_j(x)$ and $\w\vp_j(x)\sim\sum_{\ell\in\Z}\cos^2\(\frac{\pi\ell}{2^{j+1}}\)e^{{\rm i}\ell x}$ and conditions \eqref{c1}--\eqref{c3} are fulfilled with $N=0$ and $s=2$. Note that the operators of such type are used in applications, for example, in order to reduce noise contribution (see, e.g.,~\cite{ZWS12}).

\noindent 1c) At the same time if in~\eqref{A}, we replace the Dirichlet kernel $\mathcal{D}_j$ by 
$$
\mathcal{D}_j^*(x)=\sum_{\ell\in D_j}\frac{1}{\cos^2\(\frac{\pi\ell}{2^{j+1}}\)}e^{{\rm i}\ell x},
$$
then condition~\eqref{c3} will hold for arbitrary $s>0$.

\noindent 1d) We also consider the following type of operators:
\begin{equation}\label{Bj}
  B_j(f,\vp_j)(x)=2^{-j}\sum_{k\in D_j} \(f(x_k^j)+a2^{-j}f'(x_k^j)+b2^{-2j}f''(x_k^j)\)\mathcal{D}_j(x-x_k^j).
\end{equation}
We have that $B_j(f,\vp_j)=Q_j(f,\vp_j,\w\vp_j)$ if $\vp_j=\mathcal{D}_j$ and
$$
\w\vp_j(x)=\d(x)+a2^{-j}\d'(x)+b2^{-2j}\d''(x).
$$
It is not difficult to see that if $b\neq 0$, then condition~\eqref{c1} holds with $N=2$. At the same time if $a\neq 0$, then compatibility condition~\eqref{c3} holds with $s=1$. Note that in the non-periodic case operators of type~\eqref{Bj} have been studied, e.g., in~\cite{BHR93}, \cite{KKS18}.
\end{example}


\begin{example}\label{ex2} \normalfont \emph{Kantorovich-type operators} are defined by
 \begin{equation}\label{k}
  K_j(f,\vp_j)(x)=\sum_{k\in D_j} \frac{2^{\s-1}}{\pi}\int_{-\pi2^{-j-\s}}^{\pi2^{-j-\s}}f(t+x_k^j){\rm d}t\,\vp_j(x-x_k^j),
\end{equation}
where $\s\ge 1$ and as above $\vp_j\in \mathcal{T}_j$. It is clear that by taking $\w\vp_j(x)=2^{j+\s} \chi_{[-\pi2^{-j-\s}, \pi2^{-j-\s}]}(x)$, i.e., the normalized characteristic function of $[-\pi 2^{-j-\s},\pi 2^{-j-\s}]$, we have that $K_j(f,\vp_j)=Q_j(f,\vp_j,\w\vp_j)$. Next,  since
$$
\h{\w\vp_j}(\ell)=\frac{\sin \pi2^{-j-\s}\ell}{\pi2^{-j-\s}\ell},\quad \ell\in\Z,
$$
it is not difficult to see that~\eqref{c1} holds for $N=0$. Concerning condition~\eqref{c3}, we have that in this case it has the following form:
\begin{equation}\label{c3+}
  \bigg|1-\h{\vp_j}(k)\frac{\sin \pi2^{-j-\s}k}{\pi2^{-j-\s}k}\bigg|\lesssim |2^{-j}k|^s,\quad \forall k\in D_j,\quad \forall j\in\Z_+.
\end{equation}


\noindent 2a) If $\vp_j=\mathcal{D}_j$, we have that~\eqref{c3+} holds for $s=2$.

\noindent 2b) At the same time, for
$$
\vp_j(x)=\mathcal{D}_j^*(x)=\sum_{\ell\in D_j}\frac{\pi2^{-j-\s}\ell}{\sin \pi2^{-j-\s}\ell}e^{{\rm i}\ell x},
$$
condition~\eqref{c3+} holds for any $s>0$.

\smallskip

Note that in recent years, the Kantorovich type operators~\eqref{k} have been intensively studied in many works, see, e.g., \cite{BBSV, CV19, CSV20, KKS18, KS17, KS21} in the non-periodic case and~\cite{JBU02, KKS20, KP21} in the periodic case.
It is worth noting that operators of this type have several advantages over the interpolation and sampling operators.
Particularly, using the averages of a function instead of the sampled values $f(x_k^j)$ allows to deal with discontinues signals and to reduce the so-called time-jitter errors, which is an important issue  in digital image processing.

\end{example}

\section{Auxiliary results}


For $j\in \Z_+$ and $\psi\in L_1(\T)$, we define the following amalgam-type norm:
\begin{equation*}
  \|\psi\|_{\w A_{p,j}(\T)}=\sup_{\ell\in D_j}\(\sum_{\mu\in \Z}|\h{\psi}(\ell+2^j\mu)|^{p}\)^{1/p}\quad\text{if}\quad 1\le p<\infty
\end{equation*}
and
\begin{equation*}
  \|\psi\|_{\w A_{\infty,j}(\T)}=\|\psi\|_{A_\infty(\T)}\quad\text{if}\quad p=\infty.
\end{equation*}

\begin{lemma}\label{leKKS}
  Let $1\le q\le \infty$, $0\le \g<\a$, and let $(\vp_j)_{j\in\Z_+}$ and $(\w\vp_j)_{j\in\Z_+}$ be such that $\w\vp_j\in \mathcal{D}'(\T)$ and  $\vp_j\in \mathcal{T}_j^1$ for each $j\in \Z_+$. Suppose conditions \eqref{c1}, \eqref{c2}, and \eqref{c3} are fulfilled with $N\ge 0$ and $s>0$. Further suppose that

\begin{itemize}
  \item[(i)] $\a>N+1/q'$ if $q\neq 1$ and $\a\ge N$ if $q=1$

  or

  \item[(ii)] $N=0$ and $\sup_{j\in\Z_+}\|\w\vp_j\|_{\w A_{q',j}(\T)}<\infty$.
\end{itemize}
Then, for all $f\in A_q^\a(\T)$ and $j\in\Z_+$, we have
  \begin{equation}\label{leKKS1}
    \|f-Q_j(f,\vp_j,\w\vp_j)\|_{A_q^\g(\T)}\lesssim 2^{-j\min(\a-\g,s)}\|f\|_{A_q^\a(\T)}.
  \end{equation}
\end{lemma}

%
\begin{proof}
The proof of the lemma under conditions in $(i)$ can be found in~\cite[Remark~7]{KKS20}. In what follows, we prove~\eqref{leKKS1} assuming that condition~$(ii)$ holds.
We consider only the case $1\le q<\infty$. The case $q=\infty$ can be treated similarly.
First we show that
\begin{equation}\label{kb}
  \|Q_j(f,\vp_j,\w\vp_j)\|_{A_q^\g(\T)}\lesssim \|f\|_{A_q^\g(\T)}.
\end{equation}
Indeed, using the representation
\begin{equation*}
\begin{split}
  Q_j(f,\vp_j,\w\vp_j)(x)
  &=\sum_{\ell \in D_j}\h{\vp_j}(\ell)\(2^{-j}\sum_{k\in D_j} (f*\w\vp_j)(x_k^j)e^{-{\rm i}\ell x_k^j}\)e^{{\rm i}\ell x}\\
  &=\sum_{\ell \in D_j}\h{\vp_j}(\ell)\(\sum_{\nu\in \Z} \h{f}(\nu) {\h{\w\vp_j}}(\nu) 2^{-j}\sum_{k\in D_j}
  e^{2\pi {\rm i}\frac{\nu-\ell}{2^j}}\)e^{{\rm i}\ell x}\\
  &=\sum_{\ell \in D_j}\h{\vp_j}(\ell)\(\sum_{\mu\in \Z} \h{f}(\ell+2^j\mu) {\h{\w\vp_j}}(\ell+2^j\mu)\)e^{{\rm i}\ell x},
\end{split}
\end{equation*}
condition~\eqref{c2}, and H\"older's inequality, we derive
\begin{equation*}
\begin{split}
 \|Q_j(f,\vp_j,\w\vp_j)\|_{A_q^\g(\T)}^q&\lesssim\sum_{\ell\in D_j}(1+|\ell|)^{q\g}\bigg|\sum_{\mu\in \Z} \h{f}(\ell+2^j\mu) {\h{\w\vp_j}}(\ell+2^j\mu)\bigg|^q\\
 &\lesssim \sum_{\ell\in D_j}(1+|\ell|)^{q\g} \sum_{\mu\in \Z} |\h{f}(\ell+2^j\mu)|^q \(\sum_{\mu\in \Z}|{\h{\w\vp_j}}(\ell+2^j\mu)|^{q'}\)^{q/q'}\\
 &\lesssim  \sup_{\ell\in D_j}\(\sum_{\mu\in \Z}|{\h{\w\vp_j}}(\ell+2^j\mu)|^{q'}\)^{q/q'} \|f\|_{A_q^\g(\T)}^q\lesssim \|f\|_{A_q^\g(\T)}^q,
\end{split}
\end{equation*}
which gives~\eqref{kb}.

Now, we prove inequality~\eqref{leKKS1}. Let
$$
t_j(x)=\sum_{k\in D_j} \h{f}(k)e^{{\rm i}kx}.
$$
Applying~\eqref{kb}, we obtain
\begin{equation}\label{leK2}
  \begin{split}
     &\|f-Q_j(f,\vp_j,\w\vp_j)\|_{A_q^\g(\T)}\\
     &\le \|f-t_j\|_{A_q^\g(\T)}+\|t_j-Q_j(t_j,\vp,\w\vp)\|_{A_q^\g(\T)}+\|Q_j(f-t_j,\vp,\w\vp)\|_{A_q^\g(\T)}\\
     &\lesssim \|f-t_j\|_{A_q^\g(\T)}+\|t_j-Q_j(t_j,\vp,\w\vp)\|_{A_q^\g(\T)}\\
     &\lesssim 2^{-j(\a-\g)}\|f\|_{A_q^\a(\T)}+\|t_j-Q_j(t_j,\vp,\w\vp)\|_{A_q^\g(\T)}.
  \end{split}
\end{equation}
Next, using inequality~\eqref{c3}, we get
\begin{equation}\label{leK3}
  \begin{split}
    \|t_j-Q_j(t_j,\vp,\w\vp)\|_{A_q^\g(\T)}&=\bigg\|\sum_{k\in D_j}(1-\h{\vp_j}(k){\h{\w\vp_j}}(k))\h{f}(k)e^{ikx}\bigg\|_{A_q^\g(\T)}\\
    &\lesssim \bigg(\sum_{k\in D_j} (1+|k|)^{\g q}|2^{-j} k|^{sq}|\h{f}(k)|^q\bigg)^{1/q}\\
    &\lesssim 2^{-js}\bigg(\sum_{k\in D_j} (1+|k|)^{(s-(\a-\g))q}(1+|k|)^{\a q}|\h{f}(k)|^q\bigg)^{1/q}\\
    &\lesssim 2^{-\min(\a-\g,s)j}\|f\|_{A_q^\a(\T)}.
  \end{split}
\end{equation}
Finally, combining~\eqref{leK3} and~\eqref{leK2}, we arrive at~\eqref{leKKS1}.
\end{proof}

\begin{remark}\label{K}
  With regard to $(i)$ and $(ii)$ in Lemma~\ref{leKKS}, we note that the condition $(i)$ can be applied to the operators
$S_j(f,\vp_j)$ in Example~\ref{ex1}. Condition~$(ii)$, unlike condition $(i)$, allows any parameter $\a>\g$ and is especially beneficial for the Kantorovich type operators $K_j(f,\vp_j)$, see  Example~\ref{ex2}. Indeed, for $\w\vp_j(x)=2^{j+\s} \chi_{[-\pi2^{-j-\s}, \pi2^{-j-\s}]}(x)\sim\sum_{\ell\in \Z}\frac{\sin \pi2^{-j-\s}\ell}{\pi2^{-j-\s}\ell}e^{{\rm i}\ell x}$, we have
\begin{equation*}
  \sup_{j\in\Z_+}\|\w\vp_j\|_{\w A_{q',j}(\T)}=\sup_{j\in\Z_+,\,\ell\in D_j}\(\sum_{\mu\in \Z}\bigg|\frac{\sin\pi 2^{-\s}(2^{-j}\ell+\mu)}{\pi 2^{-\s}(2^{-j}\ell+\mu)}\bigg|^{q'}\)^{1/q'}<\infty
\end{equation*}
in the case $1<q'<\infty$. The case $q'=\infty$ is clear.
\end{remark}

%
%


%

\medskip

We will need the following Bernstein inequality.

\begin{lemma}\label{leB}
  Let $1\le q\le \infty$, $\min\{\a,\a+\b-\g\}>0$, and $\elll\in \Z_+^d$. Then, for any $f\in \mathcal{T}_\elll^d$, we have
  \begin{equation*}
    \|f\|_{A_q^{\a,\b}(\T^d)}\le 2^{\a|\elll|_1+(\b-\g)|\elll|_\infty}\|f\|_{A_q^\g(\T^d)}.
  \end{equation*}
\end{lemma}

\begin{proof}
The proof is similar to the proof of Lemma~2.10 in~\cite{BDSU16}.
\end{proof}

\begin{lemma}\label{lemon} (See~\cite{BDSU16})
Let $\a>0$, $\b\in \R$, $\e=\min(\a,\a+\b)>0$, and
$$
\psi(\k):=\a|\k|_1+\b|\k|_\infty,\quad \k\in \Z_+^d.
$$
Then the inequality
$$
\psi(\k)\le \psi(\k')-\e|\k'-\k|_1
$$
holds for all $\k,\k'\in \Z_+^d$ with $\k'\ge \k$ componentwise.
\end{lemma}


\begin{lemma}\label{lek}
  Let $T<1$, $r<t$, and $t\ge 0$. Then, for all $n\in \N$,
\begin{equation}\label{lek.1}
  \sum_{\k\not\in \D(n,T)}2^{-t|\k|_1+r|\k|_\infty}\lesssim \left\{
                                                             \begin{array}{ll}
                                                               2^{-\(t-r-(tT-r)\frac{d-1}{d-T}\)n}n^{d-1}, & \hbox{$T\ge \frac rt$,} \\
                                                               2^{-(t-r)n}, & \hbox{$T<\frac rt$}
                                                             \end{array}
                                                           \right.
\end{equation}
and
\begin{equation}\label{lek.2}
  \sup_{\k\not\in \D(n,T)}2^{-t|\k|_1+r|\k|_\infty}\lesssim \left\{
                                                             \begin{array}{ll}
                                                               2^{-\(t-r-(tT-r)\frac{d-1}{d-T}\)n}, & \hbox{$T\ge \frac rt$,} \\
                                                               2^{-(t-r)n}, & \hbox{$T<\frac rt$.}
                                                             \end{array}
                                                           \right.
\end{equation}
\end{lemma}

\begin{proof}
  Estimate~\eqref{lek.1} can be found  in the proof of~\cite[Theorem~4]{Kn00}. Estimate~\eqref{lek.2} can be proved by standard arguments using the method of Lagrange multipliers and Kuhn-Tacker conditions.
\end{proof}

\section{Littlewood-Paley type characterizations}

\begin{proposition}\label{pr1}
  Let $1\le q\le \infty$, $\a>0$, $\b\in \R$, $\a+\b>0$, and let $Q=(Q_j(\cdot,\vp_j,\w\vp_j))_{j\in \Z_+}$, where $(\vp_j)_{j\in \Z_+}$ and $(\w\vp_j)_{j\in \Z_+}$ be such that $\w\vp_j\in \mathcal{D}'(\T)$ and $\vp_j\in \mathcal{T}_j^1$ for each $j\in \Z_+$. Suppose conditions \eqref{c1}, \eqref{c2}, and \eqref{c3} are satisfied with the parameters $N\ge 0$ and $s>\max(\a+\b,\a)$.  Assume also that

  \begin{itemize}
    \item[$(i)$] $\min(\a+\b,\a)>N+1/q'$

    or

    \item[$(ii)$] $N=0$ and $\sup_{j\in\Z_+}\|\w\vp_j\|_{\w A_{q',j}(\T)}<\infty$.
  \end{itemize}
  Then every function $f\in A_q^{\a,\b}(\T^d)$ can be represented by the series
  \begin{equation}\label{pr1.1}
    f=\sum_{\j\in \Z_+^d} \eta_\j^Q(f),
  \end{equation}
  which converges unconditionally in $A_q^{\w\a,\b}(\T^d)$ with $0\le\w\a<\a$ and satisfies
  \begin{equation}\label{pr1.2}
    \bigg(\sum_{\j\in \Z_+^d} 2^{q(\a|\j|_1+\b|\j|_\infty)}\|\eta_\j^Q(f)\|_{A_q(\T^d)}^q\bigg)^{1/q}\lesssim \|f\|_{A_q^{\a,\b}(\T^d)}.
  \end{equation}
\end{proposition}

\begin{proof}
\emph{Step 1.} First we prove the proposition assuming that the condition in $(i)$ holds.
Let $f\in A_q^{\a,\b}(\T^d)$ and $\j\in \Z_+^d$. We have
\begin{equation}\label{f2}
  f(\x)=\sum_{\elll\in \Z^d} \d_{\j+\elll}(f)(\x),
\end{equation}
where we set $\d_{\j+\elll}(f)=0$ for $\j+\elll\in \Z^d \setminus \Z_+^d$.
In light of~\eqref{f2}, we get
\begin{equation*}
  |\eta_\j^Q(f)(\x)|\le \sum_{\elll\in \Z^d} |\eta_\j^Q\(\d_{\j+\elll}(f)\)(\x)|
\end{equation*}
and, therefore,
\begin{equation}\label{f3}
  \|\eta_\j^Q(f)\|_{A_q(\T^d)}\le \sum_{\elll\in \Z^d} \|\eta_\j^Q\(\d_{\j+\elll}(f)\)\|_{A_q(\T^d)}.
\end{equation}
In what follows, for simplicity we consider only the case $q<\infty$. The case $q=\infty$ can be treated similarly.
Multiplying by $2^{\a|\j|_1+\b|\j|_\infty}$ and taking $\ell_q$-norm on both sides of~\eqref{f3}, we obtain
\begin{equation}\label{f4}
  \begin{split}
    \bigg(\sum_{\j\in \Z_+^d} &2^{q(\a|\j|_1+\b|\j|_\infty)}\|\eta_\j^Q(f)\|_{A_q(\T^d)}^q\bigg)^{1/q}\\
    &\le \sum_{\elll\in \Z^d}\bigg(\sum_{\j\in \Z_+^d} 2^{q(\a|\j|_1+\b|\j|_\infty)}\|\eta_\j^Q(\d_{\j+\elll}(f))\|_{A_q(\T^d)}^q\bigg)^{1/q}\\
    &=\sum_{\elll_2^d\in \Z^{d-1}}\sum_{-j_1\le \ell_1<-1}(\dots)+\sum_{\elll_2^d\in \Z^{d-1}}\sum_{\ell_1\ge -1}(\dots)=S_1+S_2,
  \end{split}
\end{equation}
where $\elll_k^d=(\ell_k,\dots,\ell_d)$, $k=2,\dots,d$.

Consider the sum $S_1$. Denoting
$$
\eta_{\j_k^d}^Q=\eta_{\j_k^d,k}^Q=\prod_{i=k}^d (Q_{j_i}^i-Q_{j_i-1}^i),\quad k=2,\dots,d,
$$
where $Q_{j_i}^i$ is the univariate operator $Q_{j_i}(\cdot,\vp_{j_i},\w\vp_{j_i})$ acting on functions in the variable $x_i$,
we obtain
\begin{equation}\label{f5}
  \begin{split}
    S_1&=\sum_{\elll_2^d\in \Z^{d-1}}\sum_{-j_1\le \ell_1<-1}\bigg(\sum_{\j\in \Z_+^d} 2^{q(\a|\j|_1+\b|\j|_\infty)}\|(Q_{j_1}^1-Q_{j_1-1}^1)\eta_{\j_2^d}^Q(\d_{\j+\elll}(f))\|_{A_q(\T^d)}^q\bigg)^{1/q}\\
    &\le \sum_{b\in \{-1,0\}}\sum_{\elll_2^d\in \Z^{d-1}}\sum_{-j_1\le\ell_1<-1}\bigg(\sum_{\j\in \Z_+^d} 2^{q(\a|\j|_1+\b|\j|_\infty)}\|(Q_{j_1+b}^1-I)\eta_{\j_2^d}^Q(\d_{\j+\elll}(f))\|_{A_q(\T^d)}^q\bigg)^{1/q},\\
  \end{split}
\end{equation}
where $I$ is the identity operator.
Taking into account that
$
Q_j(t,\vp,\w\vp)=\w\vp_j*\vp_j*t
$
for any trigonometric polynomial $t\in \mathcal{T}_{j-1}^1$ and using condition~\eqref{c3} and Bernstein's inequality, we derive that
\begin{equation}\label{f6}
  \begin{split}
     \sum_{\j\in \Z_+^d} &2^{q(\a|\j|_1+\b|\j|_\infty)}\|(Q_{j_1+b}^1-I)\eta_{\j_2^d}^Q(\d_{\j+\elll}(f))\|_{A_q(\T^d)}^q\\
     &\lesssim\sum_{\j\in \Z_+^d} 2^{q(\a|\j|_1+\b|\j|_\infty-sj_1)}\bigg\|\frac{\partial^s}{\partial x_1^s}\eta_{\j_2^d}^Q(\d_{\j+\elll}(f))\bigg\|_{A_q(\T^d)}^q\\
     &\lesssim \sum_{\j\in \Z_+^d} 2^{q(\a|\j|_1+\b|\j|_\infty+s\ell_1)}\|\eta_{\j_2^d}^Q(\d_{\j+\elll}(f))\|_{A_q(\T^d)}^q\\
     &\lesssim 2^{(s-\a)\ell_1 q} \sum_{\j\in \Z_+^d} 2^{q(\a|\j+\ell_1\ee_1|_1+\b|\j|_\infty)}\|\eta_{\j_2^d}^Q(\d_{\j+\elll}(f))\|_{A_q(\T^d)}^q\\
     &\lesssim 2^{(s-\max(\a+\b,\a))\ell_1 q} \sum_{\j\in \Z_+^d} 2^{q(\a|\j+\ell_1\ee_1|_1+\b|\j+\ell_1\ee_1|_\infty)}\|\eta_{\j_2^d}^Q(\d_{\j+\elll}(f))\|_{A_q(\T^d)}^q,\\
  \end{split}
\end{equation}
where in the last inequality we use the estimates
$$
|\j|_\infty\le |\j+\ell_1\ee_1|_\infty+|\ell_1\ee_1|_\infty=|\j+\ell_1\ee_1|_\infty-\ell_1\quad\text{in the case}\quad \b\ge 0
$$
and
$$
|\j|_\infty\ge |\j+\ell_1\ee_1|_\infty\quad\text{in the case}\quad \b< 0.
$$
Next, combining~\eqref{f5} and~\eqref{f6} and using the fact that $\sum_{\ell_{1}<-1} 2^{(s-\max(\a+\b,\a))\ell_1} <\infty$, we obtain
\begin{equation}\label{f7}
    S_1\lesssim \sum_{\elll_2^d\in \Z^{d-1}}\bigg(\sum_{\j\in \Z_+^d} 2^{q(\a|\j|_1+\b|\j|_\infty)}\|\eta_{\j_2^d}^Q(\d_{j_1,j_2+\ell_2,\dots,j_d+\ell_d}(f))\|_{A_q(\T^d)}^q\bigg)^{1/q}.
\end{equation}

Now, we consider the sum $S_2$. Similar to~\eqref{f5}, we have
\begin{equation}\label{f8}
  \begin{split}
    S_2&=\sum_{\elll_2^d\in \Z^{d-1}}\sum_{\ell_1\ge -1}\bigg(\sum_{\j\in \Z_+^d} 2^{q(\a|\j|_1+\b|\j|_\infty)}\|(Q_{j_1}^1-Q_{j_1-1}^1)\eta_{\j_2^d}^Q(\d_{\j+\elll}(f))\|_{A_q(\T^d)}^q\bigg)^{1/q}\\
    &\le \sum_{b\in \{-1,0\}}\sum_{\elll_2^d\in \Z^{d-1}}\sum_{\ell_1\ge -1}\bigg(\sum_{\j\in \Z_+^d} 2^{q(\a|\j|_1+\b|\j|_\infty)}\|(Q_{j_1+b}^1-I)\eta_{\j_2^d}^Q(\d_{\j+\elll}(f))\|_{A_q(\T^d)}^q\bigg)^{1/q}.\\
  \end{split}
\end{equation}
Choosing $\zeta$ such that  $N+1/q'<\zeta<\min(\a,\a+\b)$ and applying Lemma~\ref{leKKS}$(i)$, Bernstein's inequality, and Lemma~\ref{lemon}, we obtain for $\ell_1\ge 0$ that
\begin{equation}\label{f9}
  \begin{split}
     &\sum_{\j\in \Z_+^d} 2^{q(\a|\j|_1+\b|\j|_\infty)}\|(Q_{j_1+b}^1-I)\eta_{\j_2^d}^Q(\d_{\j+\elll}(f))\|_{A_q(\T^d)}^q\\
     &\lesssim \sum_{\j\in \Z_+^d} 2^{q(\a|\j|_1+\b|\j|_\infty-\zeta j_1)}\|\eta_{\j_2^d}^Q(\d_{\j+\elll}(f))\|_{A_q^{(\zeta,0,\dots,0)}(\T^d)}^q\\
     &\lesssim \sum_{\j\in \Z_+^d} 2^{q(\a|\j|_1+\b|\j|_\infty+\zeta \ell_1)}\|\eta_{\j_2^d}^Q(\d_{\j+\elll}(f))\|_{A_q^{}(\T^d)}^q\\
     &\lesssim 2^{-q(\min(\a,\a+\b)-\zeta)\ell_1}\sum_{\j\in \Z_+^d} 2^{q(\a|\j+\ell_1\ee_1|_1+\b|\j+\ell_1\ee_1|_\infty)}\|\eta_{\j_2^d}^Q(\d_{\j+\elll}(f))\|_{A_q^{}(\T^d)}^q\\
     &\lesssim 2^{-q(\min(\a,\a+\b)-\zeta)\ell_1}\sum_{\j\in \Z_+^d} 2^{q(\a|\j|_1+\b|\j|_\infty)}\|\eta_{\j_2^d}^Q(\d_{j_1,j_2+\ell_2,\dots,j_d+\ell_d}(f))\|_{A_q^{}(\T^d)}^q.\\
  \end{split}
\end{equation}
A similar estimate clearly holds for $\ell_1=-1$. Thus, combining~\eqref{f8} and~\eqref{f9} and taking into account that
$\sum_{\ell_1\ge -1} 2^{-(\min(\a,\a+\b)-\zeta)\ell_1}<\infty$, we get
\begin{equation}\label{f10}
    S_2\lesssim \sum_{\elll_2^d\in \Z^{d-1}}\bigg(\sum_{\j\in \Z_+^d} 2^{q(\a|\j|_1+\b|\j|_\infty)}\|\eta_{\j_2^d}^Q(\d_{j_1,j_2+\ell_2,\dots,j_d+\ell_d}(f))\|_{A_q(\T^d)}^q\bigg)^{1/q}.
\end{equation}
In the next step, collecting~\eqref{f4}, \eqref{f7}, and~\eqref{f10} implies
\begin{equation*}
\begin{split}
\bigg(\sum_{\j\in \Z_+^d} &2^{q(\a|\j|_1+\b|\j|_\infty)}\|\eta_\j^Q(f)\|_{A_q(\T^d)}^q\bigg)^{1/q}\\
 &\lesssim \sum_{\elll_2^d\in \Z^{d-1}}\bigg(\sum_{\j\in \Z_+^d} 2^{q(\a|\j|_1+\b|\j|_\infty)}\|\eta_{\j_2^d}^Q(\d_{j_1,j_2+\ell_2,\dots,j_d+\ell_d}(f))\|_{A_q(\T^d)}^q\bigg)^{1/q}.
\end{split}
\end{equation*}
Then, repeating the above procedure for the parameters $\ell_2,\dots,\ell_d$, we prove~\eqref{pr1.2} by Lemma~\ref{le1}.

\medskip

\emph{Step 2.} Let us prove representation~\eqref{pr1.1}. Applying Lemma~\ref{leB} (here without loss of generality, we can assume that $\min(\w\a+\b,\w\a)>0$), H\"older's inequality, and~\eqref{pr1.2}, we obtain
\begin{equation}\label{f12}
  \begin{split}
    &\sum_{\k\in\Z_+^d}\|\eta_\k^Q(f)\|_{A_q^{\w\a,\b}(\T^d)}\\
    &\lesssim \sum_{\k\in\Z_+^d}2^{\w\a|\k|_1+\b|\k|_\infty}\|\eta_\k^Q(f)\|_{A_q(\T^d)}\\
    &=\sum_{\k\in\Z_+^d}2^{-(\a-\w\a)|\k|_1}\cdot 2^{\a|\k|_1+\b|\k|_\infty}\|\eta_\k^Q(f)\|_{A_q(\T^d)}\\
    &\le \(\sum_{\k\in\Z_+^d}2^{-q'(\a-\w\a)|\k|_1}\)^{1/q'}
    \(\sum_{\k\in\Z_+^d}2^{q(\a|\k|_1+\b|\k|_\infty)}\|\eta_\k^Q(f)\|_{A_q(\T^d)}^q\)^{1/q}\\
    &\lesssim \|f\|_{A_q^{\a,\b}(\T^d)}.
  \end{split}
\end{equation}
Therefore, $\sum_{\k\in\Z_+^d}\eta_\k^Q(f)$ converges unconditionally in $A_q^{\w\a,\b}(\T^d)$.

Now we show that for any trigonometric polynomial $g$,
\begin{equation}\label{f13}
  g=\sum_{\k\in \Z_+^d}\eta_\k^Q(g).
\end{equation}
It is clear that it suffices to verify~\eqref{f13} for $t_\j(\x)=e^{{\rm i}(\j,\x)}$ with arbitrary $\j\in \Z^d$.
By the triangle inequality, we have
\begin{equation}\label{f14}
  \begin{split}
     &\bigg\|t_\j-\sum_{\k\in \Z_+^d}\eta_\k^Q(t_\j)\bigg\|_{A_q^{\w\a,\b}(\T^d)}\\
     &\le \bigg\|t_\j-\sum_{\k\in \Z_+^d,\, |\k|_\infty\le m}\eta_\k^Q(t_\j)\bigg\|_{A_q^{\w\a,\b}(\T^d)}+\sum_{\k\in \Z_+^d,\, |\k|_\infty> m}\|\eta_\k^Q(t_\j)\|_{A_q^{\w\a,\b}(\T^d)}\\
     &:=I_1(m)+I_2(m).
  \end{split}
\end{equation}
We obviously have that
\begin{equation}\label{f15}
  I_2(m)=0\quad\text{for $m$ large enough}.
\end{equation}
Since
$$
\sum_{k_i=0}^m(Q_{k_i}^i-Q_{k_i-1}^i)=Q_m^i,\quad i=1,\dots,m,
$$
we obtain
$$
\sum_{\k\in \Z_+^d,\, |\k|_\infty\le m} \eta_\k^Q =\prod_{i=1}^d Q_m^i.
$$
Thus, for any $\j \in (-2^{m-1},2^{m-1})^d\cap \Z^d$, we have
$$
\sum_{\k\in \Z_+^d,\, |\k|_\infty\le m} \eta_\k^Q(t_\j)(\x)=\prod_{i=1}^d \h{\vp_m}(j_i){\h{\w\vp_m}}(j_i)e^{{\rm i}(\j,\x)}.
$$
Using this and conditions \eqref{c1}, \eqref{c2}, \eqref{c3}, we find
\begin{equation}\label{f16}
  \begin{split}
    I_1(m)&=\bigg|1-\prod_{i=1}^d \h{\vp_m}(j_i){\h{\w\vp_m}}(j_i)\bigg|\\
  &=\bigg|1-\h{\vp_m}(j_1){\h{\w\vp_m}}(j_1)+\sum_{\nu=2}^d \prod_{i=1}^{\nu-1} \h{\vp_m}(j_i){\h{\w\vp_m}}(j_i) \(1-\h{\vp_m}(j_\nu){\h{\w\vp_m}}(j_\nu)\)\bigg|\\
  &\lesssim \sum_{\nu=1}^d \Big|1-\h{\vp_m}(j_\nu){\h{\w\vp_m}}(j_\nu)\Big|\lesssim 2^{-ms}\sum_{\nu=1}^d|j_\nu|^s\to 0\quad \text{as}\quad m\to \infty.
  \end{split}
\end{equation}
Therefore, combining~\eqref{f14}, \eqref{f15}, and~\eqref{f16}, we arrive at~\eqref{f13}.

The rest of the proof is quite standard. Denote $F:=\sum_{\k\in\Z_+^d}\eta_\k^Q(f)$. Using~\eqref{f13}, we have for every trigonometric polynomial $g$ that
\begin{equation*}
  F-g=\sum_{\k\in \Z_+^d}\eta_\k^Q(f-g)
\end{equation*}
with convergence in $A_q^{\w\a,\b}(\T^d)$. Hence, by~\eqref{f12}, we derive
\begin{equation*}
\begin{split}
  \|F-f\|_{A_q^{\w\a,\b}(\T^d)}&\le \|F-g\|_{A_q^{\w\a,\b}(\T^d)}+\|g-f\|_{A_q^{\w\a,\b}(\T^d)}\lesssim \|f-g\|_{A_q^{\a,\b}(\T^d)}.
\end{split}
\end{equation*}
Choosing $g$ close enough to $f$ yields $\|F-f\|_{A_q^{\w\a,\b}(\T^d)}<\e$ for all $\e>0$ and hence $\|F-f\|_{A_q^{\w\a,\b}(\T^d)}=0$, which implies~\eqref{pr1.1}.

By the same scheme, using Lemma~\ref{leKKS} $(ii)$, the proof of the proposition under condition~$(ii)$ also follows.
\end{proof}

We will also need the following modification of inequality~\eqref{pr1.2}.

\begin{lemma}\label{le2}
  Let $f\in A_q(\T^d)$, $1\le q<p\le \infty$, and $1\le \t\le \infty$. Under conditions of Proposition~\ref{pr1},
  there exists a constant $C=C(\a,\b,q,\t,d)>0$ such that
    \begin{equation}\label{le2.1}
    \begin{split}
           \bigg(\sum_{\j\in \Z_+^d} 2^{\t(\a|\j|_1+\b|\j|_\infty)}&\|\eta_\j^Q(f)\|_{A_q(\T^d)}^\t\bigg)^{1/\t}\\
           &\le C\bigg(\sum_{\j\in \Z_+^d} 2^{\t((\a+\frac1q-\frac1p)|\j|_1+\b|\j|_\infty)}\|\d_\j(f)\|_{A_p(\T^d)}^\t\bigg)^{1/\t}
    \end{split}
  \end{equation}
whenever the sum in the right-hand side is finite.
\end{lemma}

\begin{proof}
  First, repeating the same procedure as in the proof of Step~1 of Proposition~\ref{pr1}, we obtain
  \begin{equation*}
               \bigg(\sum_{\j\in \Z_+^d} 2^{\t(\a|\j|_1+\b|\j|_\infty)}\|\eta_\j^Q(f)\|_{A_q(\T^d)}^\t\bigg)^{1/\t}\\
           \lesssim\bigg(\sum_{\j\in \Z_+^d} 2^{\t(\a|\j|_1+\b|\j|_\infty)}\|\d_\j(f)\|_{A_q(\T^d)}^\t\bigg)^{1/\t}.
  \end{equation*}
Then, applying the inequality
  \begin{equation*}
  \|\d_\j(f)\|_{A_q(\T^d)}\lesssim 2^{(\frac1q-\frac1p)|\j|_1}\|\d_\j(f)\|_{A_p(\T^d)},
\end{equation*}
which easily follows from H\"older's inequality and the fact that $\spec \d_\j(f)\subset P_{j_1}\times\dots\times P_{j_d}$,
$P_j=\{\ell \in \Z\,:\, 2^{j-1}\le |\ell|<2^j\}$, we arrive at~\eqref{le2.1}.
\end{proof}


A reverse statement to Proposition~\ref{pr1} is written as follows.

\begin{proposition}\label{pr2}
Let $\a>0$, $\b\in \R$, $\a+\b>0$, $1\le q\le \infty$, and let $(f_\j)_{\j\in \Z_+^d}$ be such that $f_\j\in \mathcal{T}_\j^d$ and
\begin{equation*}
  \(\sum_{\j\in\Z_+^d}2^{q(\a|\j|_1+\b|\j|_\infty)}\|f_\j\|_{A_q(\T^d)}^q\)^{1/q}<\infty.
\end{equation*}
Suppose that the series $\sum_{\j\in\Z_+^d}f_\j$ converges to a function $f$ in $A_q(\T^d)$. Then $f\in A_q^{\a,\b}(\T^d)$ and moreover, there is a constant $C=C(\a,\b,q,d)$ such that
\begin{equation*}
  \|f\|_{A_q^{\a,\b}(\T^d)}\le C\(\sum_{\j\in\Z_+^d} 2^{q(\a|\j|_1+\b|\j|_\infty)}\|f_\j\|_{A_q(\T^d)}^q\)^{1/q}.
\end{equation*}
\end{proposition}

\begin{proof}
  The proposition can be proved repeating step by step the proof of Proposition~3.4 in~\cite{BDSU16}. For completeness we present a detailed proof.

  For $\elll\in\Z_+^d$, we write $f$ as the series
  $$
   f=\sum_{\j\in\Z^d} f_{\elll+\j}
  $$
  with $f_{\elll+\j}:=0$ for $\j+\elll\in\Z^d\setminus\Z_+^d$. Using the triangle inequality and taking into account that $\d_\elll(f_{\elll+\j})=0$ for $\j\not\in \Z_+^d$, we obtain
  \begin{equation*}
  \begin{split}
         \|\d_\elll(f)\|_{A_q(\T^d)}&=\bigg\|\sum_{\j\in\Z_+^d}\d_\elll(f_{\elll+\j})\bigg\|_{A_q(\T^d)}\\
         &\le \sum_{\j\in\Z_+^d}\|\d_\elll(f_{\elll+\j})\|_{A_q(\T^d)}\le \sum_{\j\in\Z_+^d}\|f_{\elll+\j}\|_{A_q(\T^d)}.
  \end{split}
  \end{equation*}
This inequality together with Lemma~\ref{lemon} yields
\begin{equation*}
  2^{\a|\elll|_1+\b|\elll|_\infty}\|\d_\elll(f)\|_{A_q(\T^d)}\lesssim \sum_{\j\in\Z_+^d} 2^{-\min\{\a,\a+\b\}|\j|_1}\cdot 2^{\a|\elll+\j|_1+\b|\elll+\j|_\infty}\|f_{\elll+\j}\|_{A_q(\T^d)}.
\end{equation*}
Then, by Minkowski's inequality, we obtain
\begin{equation*}
  \begin{split}
      &\(\sum_{\elll\in\Z_+^d}2^{q(\a|\elll|_1+\b|\elll|_\infty)}\|\d_\elll(f)\|_{A_q(\T^d)}^q\)^{1/q}\\
      &\lesssim \sum_{\j\in\Z_+^d} 2^{-\min\{\a,\a+\b\}|\j|_1} \(\sum_{\elll\in\Z_+^d}2^{q(\a|\elll+\j|_1+\b|\elll+\j|_\infty)}\|f_{\elll+\j}\|_{A_q(\T^d)}^q\)^{1/q}\\
      &\lesssim \(\sum_{\elll\in\Z_+^d}2^{q(\a|\elll|_1+\b|\elll|_\infty)}\|f_{\elll}\|_{A_q(\T^d)}^q\)^{1/q}.
  \end{split}
\end{equation*}
Thus, Lemma~\ref{le1} concludes the proof.
\end{proof}


Propositions~\ref{pr1} and~\ref{pr2} suggest the following useful necessary and sufficient conditions for $f\in A_q^{\a,\b}(\T^d)$ to be  represented as $f=\sum_{\j\in\Z_+^d}\eta_j^Q(f)$. This generalizes Theorem~3.6 in~\cite{BDSU16}.

\begin{theorem}\label{th0}
    Let $1\le q\le \infty$, $\a>0$, $\b\in \R$, $\a+\b>0$, and let $Q=(Q_j(\cdot,\vp_j,\w\vp_j))_{j\in \Z_+}$, where $(\vp_j)_{j\in \Z_+}$ and $(\w\vp_j)_{j\in \Z_+}$ be such that $\w\vp_j\in \mathcal{D}'(\T)$ and $\vp_j\in \mathcal{T}_j^1$ for each $j\in \Z_+$. Suppose conditions \eqref{c1}, \eqref{c2}, and \eqref{c3} are satisfied with parameters $N\ge 0$ and $s>\max(\a+\b,\a)$.  Assume also that
  \begin{itemize}
    \item[$(i)$] $\min(\a+\b,\a)>N+1/q'$

    or

    \item[$(ii)$] $N=0$ and $\sup_{j\in\Z_+}\|\w\vp_j\|_{\w A_{q',j}(\T)}<\infty$.
  \end{itemize}
  Then a function $f$  belongs to $A_q^{\a,\b}(\T^d)$ if and only if it can be represented by the series~\eqref{pr1.1} converging unconditionally in $A_q^{\w\a,\b}(\T^d)$ with $\w\a<\a$ and satisfying $\sum_{\j\in \Z_+^d} 2^{q(\a|\j|_1+\b|\j|_\infty)}\|\eta_\j^Q(f)\|_{A_q(\T^d)}^q<\infty$. Moreover, the norm $\|f\|_{A_q^{\a,\b}(\T^d)}$ is equivalent to the norm 
$$
\|f\|_{A_q^{\a,\b}(\T^d)}^+:=\bigg(\sum_{\j\in \Z_+^d} 2^{q(\a|\j|_1+\b|\j|_\infty)}\|\eta_\j^Q(f)\|_{A_q(\T^d)}^q\bigg)^{1/q}.
$$
\end{theorem}

%

\section{Error estimates}

In this section, we obtain estimates for the error of approximation by quasi-interpolation operators
$$
P_{n,T}^Q=\sum_{\j\in \D(n,T)}\eta_\j^Q, \quad n\in \N,\quad T<1,
$$
where
\begin{equation*}
 \D(n,T)=\{\k\in \Z_+^d\,:\, |\k|_1-T|\k|_\infty\le (1-T)n\}.
\end{equation*}
In what follows, we distinguish between approximation of a function $f\in A_p^{\a,\b}(\T^d)$ in the isotropic space $A_q^\g(\T^d)$ (Theorem~\ref{th1}) and in the mixed space $A_{q,\mix}^\g(\T^d)$ (Theorem~\ref{th1+}) since we use slightly different ingredients in the corresponding proofs.

Recall that
$$
\s_{p,q}=\(\frac1q-\frac1p\)_+.
$$

\subsection{Error estimates in $A_{q}^\g(\T^d)$}

\begin{theorem}\label{th1}
  Let $1\le p,q\le\infty$, $\a>\s_{p,q}$, $\b\in \R$, $\g\ge 0$,   $\g-\b<\a-\s_{p,q}$, and let $Q=(Q_j(\cdot,\vp_j,\w\vp_j))_{j\in \Z_+}$, where $\w\vp_j\in \mathcal{D}'(\T)$ and $\vp_j\in \mathcal{T}_j^1$ for each $j\in \Z_+$. Suppose conditions \eqref{c1}, \eqref{c2}, and \eqref{c3} are satisfied with parameters $N\ge 0$ and $s>\max(\a+\b,\a)$. Assume also that
  \begin{itemize}
    \item[$(i)$] $\min(\a+\b,\a)>N+1/p'$

    or

    \item[$(ii)$] $N=0$ and $\sup_{j\in\Z_+}\|\w\vp_j\|_{\w A_{q',j}(\T)}<\infty$.
  \end{itemize}
  Then,  for all $f\in A_p^{\a,\b}(\T^d)$ and $n\in \N$, we have
  \begin{equation}\label{th1.1}
\begin{split}
       \|f-P_{n,T}^Q f\|_{A_q^\g(\T^d)}\le C\Omega(n)\|f\|_{A_p^{\a,\b}(\T^d)},
\end{split}
  \end{equation}
where
\begin{equation*}
  \begin{split}
      \Omega(n)
=\left\{
                                                                         \begin{array}{ll}
                                                                            \displaystyle 2^{-\(\a+\b-\g-\s_{p,q}-\((\a-\s_{p,q})T-(\g-\b)\)\frac{d-1}{d-T}\)n}n^{(d-1)(1-\frac1p)}, & \hbox{$T\ge\frac{\g-\b}{\a-\s_{p,q}}$,} \\
                                                                           \displaystyle
2^{-(\a+\b-\g-\s_{p,q})n}, & \hbox{$T<\frac{\g-\b}{\a-\s_{p,q}}$,}
                                                                          \end{array}
                                                                        \right.
   \end{split}
\end{equation*}
and the constant $C$ does not depend on $f$ and $n$.
\end{theorem}

\begin{remark}\label{remS1}
$(i)$ In the case $p=q=2$ and $Q=(I_j)_{j\in\Z_+}$, where $I_j$ is the Lagrange interpolation operator defined in~\eqref{I2}, Theorem~\ref{th1} was proved in~\cite{GH20}, see also~\cite{BDSU16} and~\cite{GH14}.  For similar results in the case $p=q=2$, $\g=T=0$, and $Q=(K_j)_{j\in\Z_+}$, where $K_j$ is defined in~\eqref{k}, see~\cite{K21}.

$(ii)$
Under conditions of Theorem~\ref{th1} with $1\le q\le 2$ and $\g=0$, by the Hausdorff-Young inequality, 
estimate~\eqref{th1.1} implies that
\begin{equation}\label{qst}
  \|f-P_{n,T}^Q f\|_{L_{q'}(\T^d)}\le C\Omega(n)\|f\|_{A_p^{\a,\b}(\T^d)}.
\end{equation}
We can further extend this result considering a more general Pitt's inequality \cite{t1,t2}
\begin{equation}\label{qst-}
\|f\|_{L_{\xi}^\eta(\T^d)}\lesssim \|f\|_{A_q^{\g}(\T^d)},\qquad 1\le q\le \xi \le\infty,
\end{equation}
where
$\|f\|_{L_{\xi}^\eta(\T^d)}=(\int_{\T^d}|f(x)|^\xi|x|^\eta dx)^{1/\xi}$
under the suitable conditions on $\xi$ and $\eta$.
In particular, using \cite[(5.4)]{t0} we have
\begin{equation}\label{qst--}
  \|f-P_{n,T}^Q f\|_{L_{\xi}(\T^d)}\le C\Omega(n)\|f\|_{A_p^{\a,\b}(\T^d)}
  \end{equation}
for $\xi\ge 2$, $\max(q,q')\le \xi$, and
$\gamma=d (1-\frac1\xi-\frac1q)$.
Taking $\xi=q'$ we see that (\ref{qst--}) coincides with
(\ref{qst}).

\end{remark}

\begin{proof}[Proof of Theorem~\ref{th1}]
  First, we consider the case $1<p\le q\le \infty$. 
  Using the estimate  $\|\cdot\|_{\ell_q}\le \|\cdot\|_{\ell_p}$, Proposition~\ref{pr1}, Lemma~\ref{leB}, and H\"older's inequality, we obtain
  \begin{equation}\label{th1.2}
  \begin{split}
          \|f-P_{n,T}^Q f\|_{A_q^\g(\T^d)}&\le \|f-P_{n,T}^Q f\|_{A_p^\g(\T^d)}\\
          &=\bigg\|\sum_{\j\not\in\D(n,T)} \eta_\j^Q(f)\bigg\|_{A_p^\g(\T^d)}\le \sum_{\j\not\in\D(n,T)} \|\eta_\j^Q(f)\|_{A_p^\g(\T^d)}\\
          &\le \sum_{\j\not\in\D(n,T)} 2^{\g |\j|_\infty}\|\eta_\j^Q(f)\|_{A_p(\T^d)}\\
          &= \sum_{\j\not\in\D(n,T)} 2^{-\a|\j|_1-(\b-\g)|\j|_\infty}2^{\a|\j|_1+\b|\j|_\infty}\|\eta_\j^Q(f)\|_{A_p(\T^d)}\\
          &\le \bigg(\sum_{\j\not\in\D(n,T)} 2^{-p'(\a|\j|_1+(\b-\g)|\j|_\infty)}\bigg)^{1/p'}\\
          &\qquad\qquad\qquad\times \bigg(\sum_{\j\not\in\D(n,T)} 2^{p(\a|\j|_1+\b|\j|_\infty)}\|\eta_\j^Q(f)\|_{A_p(\T^d)}^p\bigg)^{1/p}
  \end{split}
  \end{equation}
  Thus, Proposition~\ref{pr1} implies
    \begin{equation}\label{th1.3}
  \begin{split}
          \|f-P_{n,T}^Q f\|_{A_q^\g(\T^d)}\le \bigg(\sum_{\j\not\in\D(n,T)} 2^{-p'(\a|\j|_1+(\b-\g)|\j|_\infty)}\bigg)^{1/p'}\|f\|_{A_p^{\a,\b}(\T^d)}.
  \end{split}
\end{equation}
Next, combining~\eqref{th1.3} and~\eqref{lek.1}, we derive~\eqref{th1.1} in the case $p>1$.
The case $p=1$ is treated similarly using~\eqref{lek.2}.

Now, we consider the case $1\le q<p\le \infty$.
Since $1/p'>1/q'$, we can apply the intermediate estimate in~\eqref{th1.2} with $p=q$ given by
$$
\|f-P_{n,T}^Q f\|_{A_q^\g(\T^d)}\le \sum_{\j\not\in\D(n,T)} 2^{\g |\j|_\infty}\|\eta_\j^Q(f)\|_{A_q(\T^d)}.
$$
Then, using H\"older's inequality, Lemma~\ref{lek}, and Lemma~\ref{le2} (note that condition $(i)$ implies that $\min(\a-1/q+1/p+\b,\a-1/q+1/p)>N+1/q'$), we get
\begin{equation}\label{th1.5}
  \begin{split}
              \|f-P_{n,T}^Q f\|_{A_q^\g(\T^d)}
              &\le \(\sum_{\j\not\in\D(n,T)} 2^{-p'((\a-\s_{p,q})|\j|_1+(\b-\g)|\j|_\infty)}\)^{1/p'}\\
               &\qquad\qquad\qquad\times \(\sum_{\j\not\in\D(n,T)} 2^{p((\a-\s_{p,q})|\j|_1+\b|\j|_\infty)}\|\eta_\j^Q(f)\|_{A_q(\T^d)}^p\)^{1/p}\\
               &\lesssim \Omega(n) \(\sum_{\j\in\Z_+^d} 2^{p(\a|\j|_1+\b|\j|_\infty)}\|\d_\j(f)\|_{A_p(\T^d)}^p\)^{1/p}\\
               &\lesssim \Omega(n) \Vert f\Vert_{A_p^{\a,\b}(\T^d)},
  \end{split}
\end{equation}
where in the last estimate we have taken into account  Lemma~\ref{le1}.
\end{proof}


\subsection{Error estimates in $A_{q,\mix}^\g(\T^d)$}

\begin{theorem}\label{th1+}
  Let $1\le p,q\le\infty$, $\b\in \R$, $\g> 0$, $\g-\b+\s_{p,q}<\a$, $\g+\s_{p,q}\le \a$, and let $Q=(Q_j(\cdot,\vp_j,\w\vp_j))_{j\in \Z_+}$, where $\w\vp_j\in \mathcal{D}'(\T)$ and $\vp_j\in \mathcal{T}_j^1$ for each $j\in \Z_+$. Suppose conditions \eqref{c1}, \eqref{c2}, and \eqref{c3} are satisfied with parameters $N\ge 0$ and $s>\max(\a+\b,\a)$.  Assume also that
  \begin{itemize}
    \item[$(i)$] $\min(\a+\b,\a)>N+1/p'$

    or

    \item[$(ii)$] $N=0$ and $\sup_{j\in\Z_+}\|\w\vp_j\|_{\w A_{q',j}(\T)}<\infty$.
  \end{itemize}
  Then, for all $f\in A_p^{\a,\b}(\T^d)$ and $n\in \N$, we have
  \begin{equation}\label{th1.1+}
\begin{split}
       \|f-P_{n,T}^Q f\|_{A_{q,\mix}^\g(\T^d)}\le C\Omega_{\mix}(n)\|f\|_{A_p^{\a,\b}(\T^d)},
\end{split}
  \end{equation}
where
\begin{equation*}
  \begin{split}
      \Omega_{\mix}(n)=\left\{
                                                                         \begin{array}{ll}
                                                                   \displaystyle 2^{-\(\a+\b-\g-\s_{p,q}-\((\a-\g-\s_{p,q})T+\b\)\frac{d-1}{d-T}\)n}n^{(d-1)\s_{p,q}}, & \hbox{$T\ge\frac{-\b}{\a-\g-\s_{p,q}}$,} \\
                                                                   \displaystyle
2^{-(\a+\b-\g-\s_{p,q})n}, & \hbox{$T<\frac{-\b}{\a-\g-\s_{p,q}}$,}
                                                                          \end{array}
                                                                        \right.
   \end{split}
\end{equation*}
where the constant $C$ does not depend on $f$ and $n$.
\end{theorem}

\begin{remark}\label{rems2}
$(i)$ This result generalizes Theorem~5.1 in~\cite{BDSU16}, which corresponds to the case
$p=q=2$, $T=\b=0$, and $Q=(I_j)_{j\in\Z_+}$, where $I_j$ is defined in~\eqref{I2}.

$(ii)$ Using the inequality (see, e.g.,~\cite[Lemma~5.7]{BDSU16})
$$
\|f\|_{L_r(\T^d)}\lesssim \|f\|_{A_{2,\mix}^{\frac12-\frac1r}(\T^d)},\quad 2<r<\infty,
$$
we easily obtain that under conditions of Theorem~\ref{th1+} with $q=2$ and $\g=\frac12-\frac1r$, inequality~\eqref{th1.1+} implies the following error estimate:
  \begin{equation}\label{th1.1+++}
\begin{split}
       \|f-P_{n,T}^Q f\|_{L_r(\T^d)}\le C\w\Omega_{\mix}(n)\|f\|_{A_p^{\a,\b}(\T^d)},
\end{split}
  \end{equation}
where
\begin{equation*}
  \begin{split}
      \w\Omega_{\mix}(n)=\left\{
                                                                         \begin{array}{ll}
                                                                   \displaystyle 2^{-\(\a+\b-\w\s_{p,r}-\((\a-\w\s_{p,r})T+\b\)\frac{d-1}{d-T}\)n}n^{(d-1)\s_{p,2}}, & \hbox{$T\ge\frac{-\b}{\a-\w\s_{p,r}}$,} \\
                                                                   \displaystyle
2^{-(\a+\b-\w\s_{p,r})n}, & \hbox{$T<\frac{-\b}{\a-\w\s_{p,r}}$,}
                                                                          \end{array}
                                                                        \right.
   \end{split}
\end{equation*}
and
$$
\w\s_{p,r}=\(\frac12-\frac1r\)+\(\frac12-\frac1p\)_+.
$$
Comparing inequalities~\eqref{qst} and~\eqref{th1.1+++} with $r=q'$ and $1<q<2$, we see that~\eqref{th1.1+++} provides better approximation order in the case $1<q<2\le p\le \infty$.
\end{remark}

\begin{proof}[Proof of Theorem~\ref{th1+}]
 First, we consider the case $1\le p\le q\le \infty$.  
  Using Proposition~\ref{pr2} with
  $$
  f_\j=\left\{
         \begin{array}{ll}
           \eta_\j^Q(f), & \hbox{$\j \not\in \D(n,T)$,} \\
           0, & \hbox{$\j\in \D(n,T)$,}
         \end{array}
       \right.
  $$
and taking into account that $A_{q,\mix}^\g(\T^d)=A_q^{\g,0}(\T^d)$ and $f-P_{n,T}^Q f=\sum_{\j\in\Z_+^d} f_\j$, we obtain
 \begin{equation}\label{th2.2}
   \begin{split}
      &\|f-P_{n,T}^Q f\|_{A_{q, \mix}^\g(\T^d)}\\&\le \|f-P_{n,T}^Q f\|_{A_{p, \mix}^\g(\T^d)}
      \lesssim \(\sum_{\j\in\Z_+^d}2^{p\g |\j|_1}\|f_\j\|_{A_p(\T^d)}^p\)^{1/p}\\
      &=\(\sum_{\j\not\in\D(n,T)} 2^{-p(\a-\g)|\j|_1-p\b|\j|_\infty}2^{p(\a|\j|_1+\b|\j|_\infty)}\|\eta_\j^Q(f)\|_{A_p(\T^d)}^p\)^{1/p}\\
      &\lesssim \max_{\j\not\in\D(n,T)} 2^{-p(\a-\g)|\j|_1-p\b|\j|_\infty} \(\sum_{\j\in \Z_+^d} 2^{p(\a|\j|_1+\b|\j|_\infty)}\|\eta_\j^Q(f)\|_{A_p(\T^d)}^p\)^{1/p}\\
      &\lesssim \Omega_{\mix}(n) \|f\|_{A_{p}^{\a,\b}(\T^d)},
   \end{split}
 \end{equation}
where the last inequality follows from Proposition~\ref{pr1}.


Second, let $1\le q<p\le\infty$. Similarly to the proof of~\eqref{th1.5}, using estimates~\eqref{th2.2} with $p=q$, H\"older's inequality, and Lemmas~\ref{lek} and~\ref{le2}, we have
\begin{equation*}
  \begin{split}
    \|f-&P_{n,T}^Q f\|_{A_{q, \mix}^\g(\T^d)}\\
    &\le \(\sum_{\j\not\in\D(n,T)} 2^{q\g|\j|_1}\|\eta_\j^Q(f)\|_{A_q(\T^d)}^q\)^{1/q}\\
&\le \Bigg( \bigg(\sum_{\j\not\in\D(n,T)} 2^{-\frac{qp}{p-q}\((\a-\g-\s_{p,q})|\j|_1+\b|\j|_\infty\)}\bigg)^{1-q/p}\\
    &\qquad\qquad\qquad\times\bigg(\sum_{\j\not\in\D(n,T)} 2^{p((\a-\s_{p,q})|\j|_1+\b|\j|_\infty)}\|\eta_\j^Q(f)\|_{A_q(\T^d)}^p\bigg)^{q/p}\Bigg)^{1/q}\\
&\lesssim \Omega_{\mix}(n)\bigg(\sum_{\j\in\Z_+^d} 2^{p((\a-\s_{p,q})|\j|_1+\b|\j|_\infty)}\|\eta_\j^Q(f)\|_{A_q(\T^d)}^p\bigg)^{1/p}\\
&\lesssim \Omega_{\mix}(n)\bigg(\sum_{\j\in\Z_+^d} 2^{p(\a|\j|_1+\b|\j|_\infty)}\|\d_\j(f)\|_{A_p(\T^d)}^p\bigg)^{1/p}\\
&\lesssim \Omega_{\mix}(n)\|f\|_{A_{p}^{\a,\b}(\T^d)},
  \end{split}
\end{equation*}
which proves~\eqref{th1.1+} for $1\le q<p\le\infty$ by Lemma~\ref{le1}. 
\end{proof}


It is not difficult to see that Theorems~\ref{th1} and~\ref{th1+} can also
be  established for more general operators
$$
P_{\Gamma}^Q=\sum_{\j\in \Gamma}\eta_\j^Q,
$$
where $\Gamma$ is some arbitrary set of indices in $\Z_+^d$. More precisely, we obtain the following remarks.

\begin{remark}\label{rem.th1}
Suppose that conditions of Theorem~\ref{th1} hold with $\Gamma$ instead of $\D(n,T)$.   Then, for all $f\in A_p^{\a,\b}(\T^d)$, we have
  \begin{equation*}
\begin{split}
       \|f-P_{\Gamma}^Q f\|_{A_q^\g(\T^d)}\le C \bigg(\sum_{\j\not\in\Gamma} 2^{-p'((\a-\s_{p,q})|\j|_1+(\b-\g))|\j|_\infty)}\bigg)^{1/p'}\|f\|_{A_p^{\a,\b}(\T^d)},
\end{split}
  \end{equation*}
where the constant $C$ does not depend on $f$ and $\Gamma$.
\end{remark}

\begin{remark}\label{rem.th1+}
Suppose that conditions of Theorem~\ref{th1+} hold with $\Gamma$ instead of $\D(n,T)$.   Then, for all $f\in A_p^{\a,\b}(\T^d)$, we have
  \begin{equation*}
\begin{split}
       \|f-P_{\Gamma}^Q f\|_{A_{q,\mix}^\g(\T^d)}\le C \Omega_\Gamma\|f\|_{A_p^{\a,\b}(\T^d)},
\end{split}
  \end{equation*}
where
$$
\Omega_\Gamma=\left\{
                \begin{array}{ll}
                  \displaystyle\max_{\j\not\in\Gamma} 2^{-p(\a-\g)|\j|_1-p\b|\j|_\infty}, & \hbox{$1\le p\le q\le \infty$,} \\
                  \displaystyle\bigg(\sum_{\j\not\in\Gamma} 2^{-\frac{qp}{p-q}\((\a-\g-\s_{p,q})|\j|_1+\b|\j|_\infty\)}\bigg)^{\frac1q-\frac1p}, & \hbox{$1\le q<p\le\infty$,}
                \end{array}
              \right.
$$
and the constant $C$ does not depend on $f$ and $\Gamma$.
\end{remark}

\subsection{Sharpness}
Here, we show that inequalities~\eqref{th1.1} and~\eqref{th1.1+} given in Theorems~\ref{th1} and~\ref{th1+} are sharp for specific cases of parameters. Those cases are of special interest since they provide the best order of approximation and simultaneously are optimal with respect to the computation time (cf.~\eqref{card}).  

\begin{theorem}\label{ths}
$(i)$ Under conditions of Theorem~\ref{th1}, if $0<T<\frac{\g-\b}{\a}$, we have
    \begin{equation}\label{ths.1}
\begin{split}
       \sup_{f\in UA_p^{\a,\b}(\T^d)}\|f-P_{n,T}^Q f\|_{A_p^\g(\T^d)}\asymp 2^{-(\a+\b-\g)n}
\end{split}
  \end{equation}
for sufficiently large $n$.

$(ii)$ Under conditions of Theorem~\ref{th1+}, if $0<T<\frac{-\b}{\a-\g}$, we have
    \begin{equation*}
\begin{split}
       \sup_{f\in UA_p^{\a,\b}(\T^d)}\|f-P_{n,T}^Q f\|_{A_{p,\mix}^\g(\T^d)}\asymp 2^{-(\a+\b-\g)n}
\end{split}
  \end{equation*}
for sufficiently large $n$.
\end{theorem}

\begin{proof} In view of Theorems~\ref{th1} and~\ref{th1+}, it is enough to consider only estimates from below.
We prove only $(i)$. The assertion $(ii)$ can be treated similarly.  We follow the idea of the proof of~\cite[Theorem~6.7]{BDSU16} (see also~\cite{DU13}) taking into account the following lemma on lower estimates for linear widths (see, e.g.,~\cite[Theorem~1]{Ti60}).
\begin{lemma}\label{lelin}
  Let $L_{m+1}$ be $(m+1)$-dimensional subspace in a Banach space $X$, and let $B_{m+1}(r):=\{f\in L_{m+1}\,:\,\|f\|_X\le r\}$. Then
\begin{equation*}
  \l_m(B_{m+1}(r),X):=\inf_{\mathcal{A}_m}\sup_{f\in B_{m+1}(r)}\|f-\mathcal{A}_m f\|_X\ge r,
\end{equation*}
where infimum is taken over all continuous linear operators $A_m$ in $X$ with rank at most $m$.
\end{lemma}

We use this lemma with $X=A_p^\g(\T^d)$ and $L_{2^n+1}={\rm span}\,\{e^{{\rm i}kx_1}\,:\,k=0,\dots,2^n\}$. Let also $n_0\in \N$ be such that
${\rm rank}\, P_{n-n_0,T}^Q \le 2^n$ (we can always find such $n_0$ in view of~\eqref{card}).
For any $f\in L_{2^n+1}$, we have
\begin{equation*}
\begin{split}
    \|f\|_{A_p^{\a,\b}(\T^d)}&=\(\sum_{k_1=0}^n 2^{p(\a+\b)k_1}\|\d_{k_1,0,\dots,0}(f)\|_{A_p(\T^d)}^p\)^{1/p}\\
&\le \max_{k_1\in [0,n]}2^{(\a+\b-\g)k_1}\(\sum_{k_1=0}^n 2^{p\g k_1}\|\d_{k_1,0,\dots,0}(f)\|_{A_p(\T^d)}^p\)^{1/p}\\
&\le 2^{(\a+\b-\g)n}\|f\|_{A_p^\g(\T^d)}.
\end{split}
\end{equation*}
Thus, by choosing $r=2^{-(\a+\b-\g)n}$, we  get that $B_{2^n+1}(r)\subset UA_p^{\a,\b}(\T^d)$. Using this embedding and Lemma~\ref{lelin},  we obtain
\begin{equation*}
  \begin{split}
     \sup_{f\in UA_p^{\a,\b}(\T^d)}\|f-P_{n-n_0,T}^Q f\|_{A_p^\g(\T^d)}\ge \l_{2^n}(B_{2^n+1}(r),A_p^\g(\T^d))\ge 2^{-(\a+\b-\g)n},
   \end{split}
\end{equation*}
which implies~\eqref{ths.1}.
\end{proof}

\begin{remark}
Note that the sharpness of Theorem~\ref{th1} under certain natural restrictions on distributions $\w\vp_j$ in the case $T=0$, $\b=\g=0$, and $p=2\le q\le \infty$ follows from the proof of~\cite[Theorem~4]{K21}.
For particular cases of the parameters (mainly for the cases $T=0$, $p,q\in \{1,2\}$, $p\le q$, $\g\in \{0,1\}$, $\b=0$), the sharpness of Theorem~\ref{th1} can be also established using general estimates of linear widths (see, e.g.,~\cite{BDSU16}, \cite{NNS20}).

\end{remark}


\section{Effective error estimates}

\subsection{Energy-norm based sparse grids}
Along with the general operators $P_{n,T}^Q$, in~\cite{BDSU16} and \cite{D16} the authors studied
quasi-interpolation operators
$$
P_{\D(\xi)}^Q=\sum_{\j\in \D(\xi)}\eta_\j^Q, \quad \xi>0,
$$
with specific choice of the family $Q$, where
\begin{equation*}
 \D(\xi)=\{\k\in \Z_+^d\,:\, (\a-\s_{p,q}-\e)|\k|_1-(\g-\b-\e)|\k|_\infty\le \xi\}.
\end{equation*}
For the reader's convenience, we reformulate Theorem~\ref{th1} for $P_{\D(\xi)}^Q$ noting that $\D(\xi)$ corresponds to the set $\D(n,T)$ with $T=\frac{\g-\b-\e}{\a-\s_{p,q}-\e}$ and $n=\frac{\xi}{\a-\s_{p,q}-\g+\b}$.

\begin{corollary}\label{cor1}
Under conditions of Theorem~\ref{th1}, if  $0<\e<\g-\b<\a-\s_{p,q}$, then, for all $f\in A_p^{\a,\b}(\T^d)$ and $\xi\in \N$, we have
  \begin{equation*}
     \|f-P_{\D(\xi)}^Q f\|_{A_q^\g(\T^d)}\le C2^{-\xi}\|f\|_{A_p^{\a,\b}(\T^d)}, 
  \end{equation*}
where the constant $C$ does not depend on $f$ and $\xi$.
In particular, in the case $p=q=2$ and $\g=0$, for all $f\in H^{\a,\b}(\T^d)$ and $\xi\in \N$, we have
  \begin{equation*}
     \|f-P_{\D(\xi)}^Q f\|_{L_2(\T^d)}\le C2^{-\xi}\|f\|_{H^{\a,\b}(\T^d)}. 
  \end{equation*}
\end{corollary}

\begin{proof}
The proof directly follows from Theorem~\ref{th1} by taking $T=\frac{\g-\b-\e}{\a-\s_{p,q}-\e}$.
\end{proof}

\begin{remark}\label{re1}
$(i)$ It follows from the proof of Theorem~\ref{th1} that in the case $p=1$, the assertion of Corollary~\ref{cor1} remains true for $\e=0$.

$(ii)$ Corollary~\ref{cor1} extends Theorem 4.1 in~\cite{BDSU16}, cf.~\cite[Remark~4.4]{BDSU16}, which corresponds to the case $p=q=2$ and $Q=(I_j)_{j\in\Z_+}$, where $I_j$ is defined in~\eqref{I+}.
\end{remark}

\subsection{Smolyak grids}
In some special cases of parameters in Theorems~\ref{th1} and~\ref{th1+}, the Smolyak algorithm, i.e., the operators $P_{n,T}^Q$ with $T=0$, provides more effective error estimates with respect to the number of frequencies than the operators $P_{n,T}^Q$, $0<T<1$, which correspond to the energy norm based grids.
In particular, applying Theorem~\ref{th1} with  $T=0$ and $\b=\g$, we obtain the following corollary about approximation in the space $A_q^\b(\T^d)$.

\begin{corollary}\label{cor3}
Under conditions of Theorem~\ref{th1}, for all $f\in A_{p}^{\a,\b}(\T^d)$ and $n\in \N$, we have
  \begin{equation}\label{th3.1}
     \|f-P_{n,0}^Q f\|_{A_{q}^\b(\T^d)}\le C2^{-(\a-\s_{p,q})n}n^{(d-1)(1-\frac1p)}\|f\|_{A_{p}^{\a,\b}(\T^d)}, 
  \end{equation}
where the constant $C$ does not depend on $f$ and $n$. In particular, if $p=q=2$, we have
  \begin{equation*}
     \|f-P_{n,0}^Q f\|_{H^\b(\T^d)}\le C2^{-\a n}n^{\frac{(d-1)}2}\|f\|_{H^{\a,\b}(\T^d)}.
  \end{equation*}
\end{corollary}

\begin{remark}\label{remF}
By the same arguments as in Remark~\ref{remS1}$(ii)$, we have that inequality~\eqref{th3.1} with $1\le q\le 2$ and $\b=0$ implies that
$$
\|f-P_{n,0}^Q f\|_{L_{q'}(\T^d)}\le C2^{-(\a-\s_{p,q})n}n^{(d-1)(1-\frac1p)}\|f\|_{A_{p,\mix}^{\a}(\T^d)}.
$$
In particular, if $q=1$ and $Q=(I_j)_{j\in \Z_+}$, the above inequality generalizes~\cite[Theorem~5.6]{BDSU16} (the case $p=2$) and the main results of~\cite{Ha92} (the case $p=\infty$, which corresponds to the Korobov space).

\end{remark}

In a similar way, applying Theorem~\ref{th1+} with $T=\b=0$, we get the following result concerning approximation by $P_{n,0}^Q f$ in the space $A_{q, \mix}^\g(\T^d)$.

\begin{corollary}\label{cor2}
Under the conditions of Theorem~\ref{th1+}, for all $f\in A_{p,\mix}^{\a}(\T^d)$ and $n\in \N$, we have
  \begin{equation*}
     \|f-P_{n,0}^Q f\|_{A_{q, \mix}^\g(\T^d)}\le C2^{-(\a-\g-\s_{p,q})n}n^{(d-1)\s_{p,q}}\|f\|_{A_{p,\mix}^{\a}(\T^d)},
  \end{equation*}
where the constant $C$ does not depend on $f$ and $n$.  In particular, if $p=q=2$, we have
  \begin{equation*}
     \|f-P_{n,0}^Q f\|_{H_{\mix}^\g(\T^d)}\le C2^{-(\a-\g)n}\|f\|_{H_{\mix}^{\a}(\T^d)}.
  \end{equation*}
\end{corollary}

{\bf Acknowledgements}. We would like to thank Kristina Oganesyan for useful remarks.

                         \end{document}